\documentclass{amsart}

\usepackage{amsmath}
\usepackage{amssymb}
\usepackage{graphicx}
\usepackage{float}
\usepackage{xcolor}
\usepackage{url}
\usepackage{tikz}
\usepackage{mathrsfs}
\usepackage[normalem]{ulem}
\usepackage{comment}

\usetikzlibrary{positioning}
\usetikzlibrary{arrows.meta}
\newtheorem{theorem}{Theorem}[section]

\newtheorem{corollary}[theorem]{Corollary}
\newtheorem{claim}{Claim}
\newtheorem{step}{Step}
\newtheorem*{thmA}{Main Theorem}

\theoremstyle{definition}
\newtheorem{definition}[theorem]{Definition}

\theoremstyle{remark}

\numberwithin{equation}{section}

\newcommand{\End}{\mathrm{End}}

\newcommand{\Tnk}{T^{(n(k))}}
\newcommand{\Pnk}{P^{(k)}}
\newcommand{\Tnkp}{T^{(n(k+1))}}

\newcommand{\Geh}{\mathcal{G}}

\newcommand{\fTk}{\mathcal{T}_k}
\newcommand{\fPk}{\mathcal{P}_k}
\newcommand{\fTkp}{\mathcal{T}_{k+1}}
\newcommand{\fTkm}{\mathcal{T}_{k-1}}
\newcommand{\fPkp}{\mathcal{P}_{k+1}}
\newcommand{\fPkm}{\mathcal{P}_{k-1}}

\newcommand{\fRk}{\mathcal{R}_k}

\newcommand{\fRkm}{\mathcal{R}_{k-1}}

\newcommand{\fSk}{\mathcal{S}_k}

\newcommand{\fSkm}{\mathcal{S}_{k-1}}

\newcommand{\wkp}{w_{k+1}}

\newcommand{\vkm}{v_{k-1}}

\newcommand{\rts}{\mathbf{R}}

\newcommand{\tF}{\tilde{F}}

\newcommand{\Tom}{T^{\omega}}
\newcommand{\bfonk}{\mathbf{1}(k)}

\newcommand{\Tonk}{T^{(\bfonk)}}

\newcommand{\firsttree}{{\mathbf{0}(k)}}
\newcommand{\lasttree}{{\mathbf{1}(k)}}
\newcommand{\firsttreeplus}{\mathbf{0}(k+1)}
\newcommand{\lasttreeplus}{\mathbf{1}(k+1)}
\newcommand{\firsttreeminus}{\mathbf{0}(k-1)}
\newcommand{\lasttreeminus}{\mathbf{1}(k-1)}

\newcommand{\Tak}{T^{\firsttree}}

\newcommand{\Takp}{T^{\firsttreeplus}}

\newcommand{\pok}{p^{\lasttree}}
\newcommand{\aok}{a^{\lasttree}}
\newcommand{\bok}{b^{\lasttree}}

\newcommand{\baok}{\bar{a}^{\lasttree}}
\newcommand{\baokp}{\bar{a}^{\lasttreeplus}}

\newcommand{\bbok}{\bar{b}^{\lasttree}}
\newcommand{\bdok}{\bar{d}^{\lasttree}}
\newcommand{\bdokm}{\bar{d}^{\lasttreeminus}}
\newcommand{\beok}{\bar{e}^{\lasttree}}

\newcommand{\restri}{\upharpoonright}
\DeclareMathOperator{\diam}{diam}

\DeclareMathOperator{\interior}{int}
\renewcommand{\epsilon}{\varepsilon}
\begin{document}

\title{On entropy of pure mixing maps on dendrites}


\author{Dominik Kwietniak}
\address[D. Kwietniak]{Jagiellonian University in Kraków, Faculty of Mathematics and Computer Science, ul. Łojasiewicza 6, 30-348 Kraków, Poland.}
\curraddr{}
\email{dominik.kwietniak@uj.edu.pl}
\urladdr{https://www2.im.uj.edu.pl/DominikKwietniak/}
\thanks{}

\author{Piotr Oprocha}
\address[P.\ Oprocha]{Centre of Excellence IT4Innovations - Institute for Research and Applications of Fuzzy Modeling, University of Ostrava, 30. dubna 22, 701 03 Ostrava 1, Czech Republic}
\email{piotr.oprocha@osu.cz}

\author{Jakub Tomaszewski}
\address[J. Tomaszewski]{AGH University of Krakow, Faculty of Applied Mathematics,
al.\ Mickiewicza 30, 30-059 Krak\'ow, Poland. -- $\&$ --
University of Maryland, Department of Mathematics, William E. Kirwan Hall, 4176 Campus Dr., College Park, MD 20742 USA.}
\curraddr{}
\email{tomaszew@agh.edu.pl}


\date{}

\dedicatory{}


\begin{abstract}
For every $0<\alpha\le\infty$ we construct a continuous pure mixing map (topologically mixing, but not exact) on the Gehman dendrite with topological entropy $\alpha$. It has been previously shown by Špitalský that there are exact maps on the Gehman dendrite with arbitrarily low positive topological entropy. Together, these results show that the entropy of maps on the Gehman dendrite does not exhibit the paradoxical behaviour reported for graph maps, where the infimum of the topological entropy of exact maps is strictly smaller than the infimum of the entropy of pure mixing maps. The latter result, stated in terms of popular notions of chaos, says that for maps on graphs, lower entropy implies stronger Devaney chaos. The conclusion of this paper says that lower entropy does not force stronger chaos for maps of the Gehman dendrite.
\end{abstract}

\maketitle

\section{Introduction} The topological entropy is arguably the most important invariant in topological dynamics. A priori, the entropy $h(f)$ of a topological dynamical system $(X,f)$ may reach any value in the extended interval $[0,\infty]$ (for definitions, see the next section). However, there are classes of topological dynamical systems whose properties restrict the attainable values of topological entropy. For example, for expansive systems the entropy must be finite. For one-dimensional continua, such as the interval $[0,1]$, some dynamical properties may restrict the set of possible values of topological entropy. For example, for a topologically transitive interval map $f$ we have $h(f)\in[\log (2)/2,\infty]$. Therefore, the general problem we want to study here is the following: \emph{Given a continuum $X$ and a class $\mathcal{X}$ of continuous maps from $X$ to itself, investigate the set $h(\mathcal{X})=\{h(f):f\in\mathcal{X}\}$ of possible values of topological entropy for maps in $\mathcal{X}$.} This can be seen as an instance of Anatole Katok's \emph{flexibility program} \cite{EK, Erc}. The latter is a research programme in dynamical systems theory that is inspired by Katok's work from the 1980s--2000s. Note that Katok did not typically use the ``flexibility programme'' as a formal label in his articles (\cite{EK} is an exception). The concept of flexibility developed gradually through Katok's work, with the term becoming more commonly used to describe his approach retrospectively by the dynamical systems community.

The programme is summarised in \cite[page 633]{EK}: 

\begin{quote}
Under properly understood general restrictions within a fixed class of smooth dynamical systems, dynamical invariants, both quantitative and qualitative, take arbitrary values.    
\end{quote}

In particular, the explorations of connections between transitivity, density of the set of periodic points, and topological entropy for low-dimensional continuous maps fit into Katok's programme. These connections were a subject of intensive studies even before Katok formulated the idea of flexibility as a general research programme in dynamics. See \cite{AKLS,ABLM,Baldwin,KM, Spitalsky, VS-DCDS, Ye} and references therein.

Here, we concentrate our efforts on proving the flexibility of entropy for a particular dendrite---the Gehman dendrite \cite{Gehman}---and a particular class of continuous maps from $\Geh$ to itself---the class of pure mixing maps of $\Geh$ (maps that are mixing but not exact). 

Dendrites form a class of compact connected metric spaces (continua) that include all trees. 
A dendrite is a non-degenerate locally connected continuum without a subset homeomorphic to a simple closed curve. Like trees, dendrites have the fixed point property, are absolute retracts, and can be embedded in the plane, but they can be more complex: some dynamical phenomena appear on dendrites that are not possible on trees (e.g., see \cite{Byszewski,HM} for an example of a weakly mixing, not mixing dendrite map).
Dendrite dynamics has become a popular research area recently. 
Dynamics on dendrites serves as a transition zone between one-dimensional and higher-dimensional dynamics. Dendrites display enough complexity to exhibit some higher-dimensional phenomena while remaining analytically tractable because of their fundamentally one-dimensional nature. The Gehman dendrite seems to be the simplest one that allows such a construction. In fact, one can note that the combination of the construction presented here with some tools presented in \cite{Dominik_drzewa,Spitalsky,VS-JMAA} should lead to analogous results for any dendrite with an infinite set of endpoints.
On the other hand, there exists a zero entropy transitive map of the Ważewski's universal dendrite \cite{Byszewski,HM}.

Our interest in pure mixing maps is caused by the following observations. For a general topological dynamical systems given by a continuous map $f\colon X\to X$ acting on a compact metric space we have the following chain of implications
\begin{equation*}
\text{$f$ is exact}\implies\text{$f$  is mixing}\implies\text{$f$ is transitive},
\end{equation*}
where exactness, mixing, and transitivity are properties associated with a nontrivial global  dynamics. Furthermore, on dendrites containing a free arc, in particular, on all trees and the Gehman dendrite (see \cite{Dirbak}), transitivity implies that the set of periodic points of $f$ is dense. These observations impose
a hierarchy of chaotic properties  (variants of Devaney chaos) discussed in more detail in \cite{KM,Dominik_odcinek}.

In particular, one can argue that exact systems exhibit more complex behaviour than mixing ones. This leads to expectation that there should be  more restrictions for possible values of topological entropy of exact maps than for entropy of mixing but not exact (pure mixing) maps. Indeed, the entropy of pure mixing maps of the Cantor set can take any value in $[0,\infty]$, while exact maps always have positive entropy. However, the authors of \cite{Dominik_odcinek} have shown that for pure mixing interval maps, the set of possible values of entropy is the interval $(\log (3)/2,\infty]$, while the entropy of the exact maps achieves any value in $(\log (2)/2,\infty]$. In \cite{Dominik_drzewa} it was proved that the same paradoxical situation holds for topological trees and other spaces (see \cite{Dominik_drzewa} for more details): the set of possible values of entropy for maps that are \emph{more} chaotic in the hierarchy contains smaller values than the analogous set for \emph{less} chaotic maps.  Roughly speaking, sufficiently low entropy implies stronger chaos. This leads to a question considered here: What is the set of possible values of entropy of pure mixing maps of the Gehman dendrite?

It is known that on the Gehman dendrite the entropy of any transitive (hence, also pure mixing or exact) map of $\Geh$ must be positive; see \cite[Theorem C]{Dirbak}.
Also, the entropy of an exact map on $\Geh$ can be arbitrarily low; see \cite[Theorem A]{Spitalsky}. An easy modification of this reasoning shows that the topological entropy of the exact maps on the Gehman dendrite can take any value in $(0,\infty]$.
We will show that the set of possible values of entropy for pure mixing maps is the same.

\begin{thmA}[Theorem \ref{thm:main} below]\label{thm:B}
    For each $\alpha\in(0, \infty]$ there exists a pure mixing map $F_\alpha\colon\mathcal{G}\to\mathcal{G}$ on the Gehman dendrite such that $h(F_\alpha)=\alpha$.
\end{thmA}

Thus, there is no entropy paradox on the Gehman dendrite: The infima of the entropies of transitive maps, pure mixing maps, and exact maps are equal $0$.


\section{Preliminaries}
\subsection{The Gehman dendrite}\label{sec:setup}
A \emph{continuum} is a compact, connected metric space. A continuum is non-degenerate if it contains at least two points.
A \emph{dendrite} is a non-degenerate locally connected continuum without a subset homeomorphic to a simple closed curve. An \emph{arc} in a continuum $X$ is a set $A\subseteq X$ that is a homeomorphic copy of the interval $[0,1]$. In other words, $A$ is an arc if there is $\varphi\colon [0,1]\to X$ which is continuous, injective, and $\varphi([0,1])=A$.  

In this situation, we call points $\varphi(0)$ and $\varphi(1)$ the \emph{endpoints} of $A$.
A tree is a dendrite that can be written as a finite union of  arcs in such a way that each pair of arcs (definitely) has at most one point in common.

For a dendrite $X$, we write $\End(X)$ for the set of its endpoints (points $x\in X$ such that $X\setminus\{x\}$ is connected) and $B(X)$ for its \emph{branch points} or \emph{vertices} (points $x\in X$ such that $X\setminus\{x\}$ has at least three components). In any dendrite, $B(X)$ is always at most countable and is empty if and only if $X$ is an arc, while $\End(X)$ is always nonempty. An arc $A\subseteq X$ is a \emph{free arc} if all points in $A$ except possibly the endpoints of $A$ are not branch points  in $X$.

Recall that \emph{Gehman dendrite} $\Geh$ is a dendrite whose set of endpoints is homeomorphic to the Cantor set and whose branching points are of order $3$, that is, if $x\in \Geh$ is a branching point, then $\Geh\setminus\{x\}$ has three connected components. Gehman dendrite is unique up to a homeomorphism, see \cite[Theorem~4.1]{Charatonik}.

In the rest of the paper we will use the following notational conventions regarding the Gehman dendrite and its subtrees.
Recall that the full binary tree of height $n\ge 1$, denoted $T^{(n)}$, is the tree obtained by the inductive procedure:
Base case: $T^{(1)}$ is just the standard compact interval $[0,1]$ and its root is the point $1/2$. For $n \ge  1$, we construct $T^{(n+1)}$ as follows: take two disjoint copies of $T^{(n)}$, which we denote $T^{(n)}_0$ and $T^{(n)}_1$. We set $T^{(n+1)}$ to be the union $T^{(n)}_0\cup T^{(n)}_1\cup[0,1]$, where we identify the point $0\in  [0,1]$ with the root of $T^{(n)}_0$ and we identify the point  $1\in [0,1]$ with the root of $T^{(n)}_1$. We declare the point corresponding to $1/2\in[0,1]$ the root of $T^{(n+1)}$. For each $n\ge 1$ we label the vertices of $T^{(n)}$ with binary words in the standard way. In particular, we write $c_\lambda^{(n)}$, where $\lambda$ stands for the empty word, for the root and 
 $\End(T^{(n)})$ for the set of endpoints of $T^{(n)}$, that is,  
$\End(T^{(n)})=\{c^{(n)}_\omega:\omega\in\{0,1\}^{n}\}$. See Figure \ref{fig:T3}.

\begin{figure}[H]
    \centering
    \begin{tikzpicture}[
    level/.style={sibling distance=50mm/#1},
    level distance=12mm,
    inner node/.style={circle, fill=black, minimum size=2mm, inner sep=0pt},
    leaf node/.style={circle, fill=black, minimum size=2mm, inner sep=0pt}
]

\node[inner node, label={[label distance=1mm]right:$\lambda$}] (root) {}
    child {node[inner node, label={[label distance=1mm]left:0}] (0) {}
        child {node[inner node, label={[label distance=1mm]left:00}] (00) {}
            child {node[leaf node, label={[label distance=1mm]below:000}] (000) {}}
            child {node[leaf node, label={[label distance=1mm]below:001}] (001) {}}
        }
        child {node[inner node, label={[label distance=1mm]right:01}] (01) {}
            child {node[leaf node, label={[label distance=1mm]below:010}] (010) {}}
            child {node[leaf node, label={[label distance=1mm]below:011}] (011) {}}
        }
    }
    child {node[inner node, label={[label distance=1mm]right:1}] (1) {}
        child {node[inner node, label={[label distance=1mm]left:10}] (10) {}
            child {node[leaf node, label={[label distance=1mm]below:100}] (100) {}}
            child {node[leaf node, label={[label distance=1mm]below:101}] (101) {}}
        }
        child {node[inner node, label={[label distance=1mm]right:11}] (11) {}
            child {node[leaf node, label={[label distance=1mm]below:110}] (110) {}}
            child {node[leaf node, label={[label distance=1mm]below:111}] (111) {}}
        }
    };

\end{tikzpicture}    
    \caption{The tree $T^{(3)}$ with the standard labelling of its vertices. 
    }
\label{fig:T3}
\end{figure}
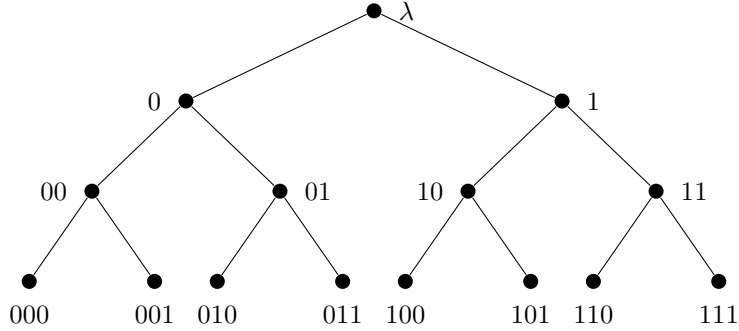

Similarly, the Gehman dendrite can be pictured as an infinite binary tree where each vertex (except the root) has exactly one parent and each vertex has exactly two children. We use finite binary words to label the vertices. First, we label the root vertex with the empty word $\lambda$, that is, we let $c_\lambda$ to be the root of $\Geh$. 
For any vertex $c_w$ of $\Geh$ labelled with binary word $w$, we label its left child as $w0$ (append $0$ to the word $w$) and we label its right child as $w1$ (append $1$ to the word $w$). That is, at level $n$, there are $2^n$ vertices $c_\omega$, each labelled with a binary word of length $n$.


\subsection{Notions from topological dynamics}\label{complexity_classification}

Let $(X,f)$ be a topological dynamical system (a TDS for short). It means that $X$ is a compact metric space and $f\colon X\to X$ is a continuous surjection. We call a closed nonempty set $A\subseteq X$ such that
$f(A)=A$ a \emph{subsystem} of $(X,f)$. If $A$ is a subsystem then $(A,f|_A)$ is a TDS.

Let $(X,f)$ and $(Y,g)$ be two TDS. We say $(Y,g)$ is a \emph{factor} of $(X,f)$ (and we call $(X,f)$ an \emph{extension} of $(Y,g)$) if there exists a \emph{factor map}, that is, a continuous surjection $\varphi\colon X\to Y$ such that $\varphi\circ f=g\circ\varphi$. If $\varphi$ is a homeomorphism, then we say that $(X,f)$ and $(Y,g)$ are \emph{conjugate}.

\begin{definition}
A TDS $(X,f)$ is: 

\begin{itemize}
\item \emph{transitive} if for every $U,V\subseteq X$ nonempty and open there is $n\in\mathbb{N}$ such that $f^{n}(U)\cap V\neq\emptyset$;
    \item 
\emph{(topologically) mixing} if for every  $U,V\subseteq X$ nonempty and open there exists $n_0\in\mathbb{N}$ such that for all $n\geq n_0$ we have $f^{n}(U)\cap V\neq\emptyset$;
\item \emph{exact} if for each open $\emptyset\neq U\subseteq X$ there is $n\in\mathbb{N}$ such that $f^{n}(U)=X$;
\item \emph{pure mixing} if it is mixing but not exact.
\end{itemize}
\end{definition}

For simplicity, we say that $f\colon X\to X$ is transitive/(pure) mixing/exact, if the TDS $(X,f)$ is (pure) mixing/exact.

By $h(f)$ we denote the \emph{topological entropy} of TDS $(X, f)$. For definition and further details, see \cite[Chapter 14]{Denker}. Topological entropy is a numerical invariant ($h(f)\in[0,\infty]$) of conjugacy of TDS. Here we will list only these properties of entropy that we need to determine the entropy of our examples.
We will use these facts without further notice.
\begin{theorem}
   The topological entropy $h(f)$ of a TDS $(X,f)$ has the following properties:
   \begin{enumerate}
       \item If $(Y, g)$ is a factor of $(X, f)$, then $h(g)\leq h(f)$. 
       If, in addition, there is $N\ge 1$ such that every point in $Y$ has at most $N$ preimages through the factor map, then $h(f)=h(g)$.
       In particular, $h(f)=h(g)$ if $f$ and $g$ are conjugate.
       \item If $A\subseteq X$ is a subsystem of $(X, f)$, then $h(f|_A)\leq h(f)$.
       \item If $n\ge 1$, then $h(f^n)=nh(f)$.
       \item Let $I$ be a nonempty set of indices. If for every $i\in I$ the set $X_i\subseteq X$ is a subsystem of $(X,f)$ and $X=\bigcup_{i\in I} X_i$, then
       \begin{equation*}
           h(f)=\sup_{i\in I}h(f|_{X_i}).
       \end{equation*}
   \end{enumerate}
\end{theorem}

To state the next result we need the definition of Hausdorff dimension. For further details, see \cite[Chapter 1]{Bishop}.
\begin{definition}
 Let $(X,d)$ be a metric space and $S\subset X$. 
 We set $$\mathcal{H}^s(S)=\lim_{\delta\to 0}\inf\{\sum_{i=1}^\infty\diam(U_i)^s\colon \bigcup_{i=1}^\infty U_i\supset S, \diam U_i<\delta\}$$ to be the \emph{$s$-dimensional Hausdorff outer measure} of $S$.
We define the \emph{Hausdorff dimension} of $S$ as $$\dim_H(S)=\inf\{s\geq 0: \mathcal{H}^s(S)=0\}=\sup\{s\geq 0:\; \mathcal{H}^s(S)=\infty\}.$$
It is well known that $\mathcal{H}^s$ restricted to Borel subsets of $X$ is a measure.
\end{definition}

\begin{theorem}[{\cite[Corollary 2.2]{Misiurewicz_hausdorff}}] \label{lipschitz_entropy_graf}
Let $(X,f)$ be a TDS on a metric space $(X,d)$. If $f\colon X\to X$ is $L$-Lipschitz for some $L>1$, that is, if $d(f(x),f(y))\le Ld(x,y)$ for every $x,y\in X$,
then $\dim_H(X)\cdot \log (L)\geq h(f)$.
\end{theorem}

Let $X$ be a continuum. A metric $d$ on $X$ is \emph{convex}
if for every distinct $x, y \in X$ there is $z\in X$ such that $d(x,z)=d(z,y)={
\frac{1}{2}}d(x,y)$. By \cite[Theorem 8]{Bing} every locally connected continuum admits a compatible convex metric. If $X$ is endowed with a convex metric $d$, then for every $a\neq b$ there is a connecting arc $A=[a, b]$, whose length satisfies $\mathcal{H}^1(A)=d(a,b)$; every such arc will be called \emph{geodesic}. We refer the reader to \cite{VS-JMAA} and references therein for a discussion on these matters.


Since every dendrite is locally connected, the Gehman dendrite always admits a convex metric. Conversely, given a nonatomic Borel probability measure $\mu$ on $\Geh$ that is positive on every free arc, we can define a convex metric $d_\mu$ on $\Geh$ by setting $d_\mu(x,y)=\mu([x,y]_\Geh)$, where $[x,y]_\Geh$ is the unique arc in $\Geh$ whose endpoints are $x$ and $y$.

\subsection{Entropy of tree maps}
Recall that a \emph{tree} is a dendrite that can be written as a finite union of arcs. 

We say that a continuous map $f\colon X\to Y$ between topological spaces is \emph{monotone} if for every $y\in Y$ the preimage $f^{-1}(y)$ is a connected subset of $X$.

A tree map $f\colon T\to T$ is \textit{$P$-monotone} if $P\subseteq T$ is a finite set  containing all vertices of $T$ such that for each connected component $C$ of $T\setminus P$ the map $f\colon \overline{C}\to T$ is monotone (here, $\overline{C}$ stands for the closure of $C$ in $T$). We call $\overline{C}$ a \emph{$P$-basic interval} of $f$. Observe that each connected component $C$ of $T\setminus P$ must be an open subset of $T$, as every vertex of $T$ belongs to $P$. Consequently, every $P$-basic interval is a free arc. 

There might be multiple finite sets $P$ such that given tree map $f$ is $P$-monotone.

Let $X$ be a continuum and $A\subseteq X$ be a free arc. We say that
a map $f\colon X\to X$ is \emph{linear on $A$} if there exists a constant $s\geq 0$ such that $\mathcal{H}^1(f(J))=s\cdot \mathcal{H}^1(J)$ for every arc $J\subset A$. We call $s$ the \emph{slope} of $f$ on $A$. We say that $f$ is \emph{piecewise linear} (\emph{$\sigma$-linear}) if $X$ can be written as a finite union (contains a countable dense union) of free arcs such that $f$ is linear on each of those arcs. Moreover, we say that a $\sigma$-linear map $f$ has \emph{constant slope} (\emph{bounded slope}/\emph{is expanding}) if for each of those arcs the slope is the same (is bounded by some $s_0$/is strictly larger than $1$). 

\begin{definition}
We say that a $P$-monotone tree map $f\colon T\to T$ is \emph{$P$-linear} if it is linear on each $P$-basic interval. Moreover, if $f(P)\subset P$, then we say that $f$ is a $P$-Markov map.
\end{definition}

We say that $f$ is \emph{piecewise monotone} (respectively, \emph{piecewise linear} or \emph{Markov}) if it is $P$-monotone ($P$-linear, $P$-Markov) for some finite $P\subset G$ and we do not need to specify the set $P$. In this setting, for piecewise linear maps without the specified $P$, we will refer to the ($P$)-basic intervals by simply calling them linearity intervals.

Fix $n\ge 1$. Given a binary word $\omega=\omega_0\ldots \omega_{n-1}\in\{0,1\}^n$ we may think of it as of a binary expansion of the integer $\langle \omega\rangle=\omega_0 2^0+\omega_1 2^1+\ldots+\omega_{n-1}2^{n-1}$. Note that the least significant digit is written first and we always use $n$ digits. We define $\omega\oplus 1$ to be the binary expansion of the integer $\langle\omega\rangle+1\; (\textrm{mod } 2^n)$. For example
\begin{equation*}
    000\oplus 1=100,\ 100\oplus 1=010,\ldots, 111\oplus 1=000, 
\end{equation*}
as depicted on Figure \ref{fig:oplus-action}. 
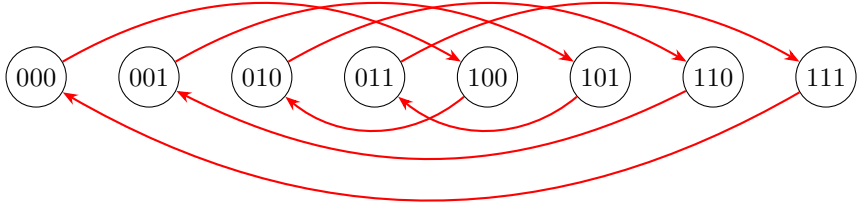
\begin{figure}[H]
    \centering

\begin{tikzpicture}[
    scale=1.5,
    node distance=2cm,
    endpoint/.style={circle, draw, minimum size=0.8cm, inner sep=0pt}
]

\node[endpoint] (000) at (0,0) {000};
\node[endpoint] (001) at (1,0) {001};
\node[endpoint] (010) at (2,0) {010};
\node[endpoint] (011) at (3,0) {011};
\node[endpoint] (100) at (4,0) {100};
\node[endpoint] (101) at (5,0) {101};
\node[endpoint] (110) at (6,0) {110};
\node[endpoint] (111) at (7,0) {111};

\draw[-{Stealth[length=2mm]}, thick, red] (000) to[bend left=30] (100);
\draw[-{Stealth[length=2mm]}, thick, red] (100) to[bend left=40] (010);
\draw[-{Stealth[length=2mm]}, thick, red] (010) to[bend left=30] (110);
\draw[-{Stealth[length=2mm]}, thick, red] (110) to[bend left=27] (001);
\draw[-{Stealth[length=2mm]}, thick, red] (001) to[bend left=30] (101);
\draw[-{Stealth[length=2mm]}, thick, red] (101) to[bend left=40] (011);
\draw[-{Stealth[length=2mm]}, thick, red] (011) to[bend left=30] (111);
\draw[-{Stealth[length=2mm]}, thick, red] (111) to[bend left=30] (000);

\end{tikzpicture}
    \caption{
     The action of $\omega\mapsto \omega\oplus 1$ operation on binary words of length $3$.}
    
    \label{fig:oplus-action}
\end{figure}

Below we reformulate
\cite[Lemma 9.2]{Dominik_drzewa} adding a corollary (Corollary \ref{cor:structureTn}) that explicitly lists some properties that follow from the proof
of Lemma 9.2 presented in \cite{Dominik_drzewa}. 
Formally, \cite[Lemma 9.2]{Dominik_drzewa} considers only a special case of Theorem \ref{odwzorowanie_drzewo_mala_entropia+}, namely only the case $\ell=1$ is considered. But the proof of \cite[Lemma 9.2]{Dominik_drzewa} can be repeated verbatim, except that if $\ell>1$ then  one must replace the initial $3$-fold (a.k.a. $3$-horseshoe) function 
by the $3^l$-fold function. 

Then one repeats the inductive proof from \cite{Dominik_drzewa} and checks that the additional claims listed in Corollary \ref{cor:structureTn} hold  (recall that $T^{(n)}$ is a topological tree defined in Section~\ref{sec:setup}).

\begin{theorem}\label{odwzorowanie_drzewo_mala_entropia+}
\cite[Lemma 9.2]{Dominik_drzewa}
For every $\epsilon>0$ and $n,\ell\ge 1$ there exists a topologically exact piecewise linear Markov map $f^{(n)}_{\epsilon,\ell}\colon T^{(n)}\to T^{(n)}$ such that
\begin{equation*}
    \frac{\ell\log (3)}{2^n}\leq h(f^{(n)}_{\epsilon,\ell})< \frac{\ell\log (3)}{2^n}+\epsilon.
\end{equation*}
\end{theorem}

\begin{corollary}\label{cor:structureTn}
The map $f^{(n)}_{\epsilon,\ell}\colon T^{(n)}\to T^{(n)}$ provided by  \cite[Lemma 9.2]{Dominik_drzewa} has the following properties:
\begin{enumerate}
\item 
There is a finite set $P^{(n)}\subseteq T^{(n)}$ such that $f$ is  $P^{(n)}$-Markov. Clearly, every point in $P^{(n)}$ is eventually periodic.
    \item Fix $0\le k\le n$. The set of vertices $\{c_\omega\in T^{(n)}:\omega\in\{0,1\}^k \}$ of $T^{(n)}$ labelled with binary words of length $k$ is a single periodic orbit of $f^{(n)}_{\epsilon,\ell}$ such that for every $\omega\in\{0,1\}^k$ we have $c_\omega\in  P^{(n)}$ and 
$f^{(n)}_{\epsilon,\ell}(c_\omega)=c_{\omega\oplus 1}$.  
\item The root $c_\lambda$ is the unique fixed point of $f^{(n)}_{\epsilon,\ell}$ and $c_\lambda\in  P^{(n)}$. 
\item There are $a^{(n)},b^{(n)},d^{(n)},e^{(n)},p^{(n)},q^{(n)}\in P^{(n)}$ such that $[a^{(n)},b^{(n)}]$ and $[d^{(n)},e^{(n)}]$ are disjoint free arcs, 
$p^{(n)}\in\interior[a^{(n)},b^{(n)}]$ and $q^{(n)}\in\interior[d^{(n)},e^{(n)}]$, $f^{(n)}_{\epsilon,\ell}(q^{(n)})=c_\lambda$ and $f^{(n)}_{\epsilon,\ell}(p^{(n)})=c^{(n)}_\omega$ for some $\omega\in\{0,1\}^n$. Without loss of generality, we may assume $\omega=0^n$.
\end{enumerate}
\end{corollary}

If $f$ is piecewise linear with the slope on each interval of linearity strictly greater than $1$, then there is a constant $s>1$ such that for every arc $J\subseteq \Geh$ we have that $s|J|<|f(J)|$ unless $f(J)$ contains a 
point from  $P^{(n)}$.

\section{Main Theorem}

The proof of the Main Theorem is based on the self-similarity of the Gehman dendrite $\Geh$, which can be decomposed into infinitely many floors, where each floor is a disjoint union of finitely many copies of a binary tree. 
More precisely, we start with any finite binary tree and attach to its endpoints copies of another fixed larger binary tree, forming the next floor of our dendrite. 
Then we get even bigger binary tree, and we attach another floor. This gives us a partition of $\Geh$ into \emph{floors} or \emph{levels} $\fTk$, each of these floors being a disjoint union of $2^{m(k)}$  copies of the same binary tree.

We will define the map on $\mathcal{G}$ by describing it on each level and then modifying it. The modification will add the mixing property as the image of any subset contained in $\mathcal{G}$ will grow with each iteration of the map, but the set $\mathrm{End}(\mathcal{G})$ will remain backward invariant preventing the constructed map from being exact.

We begin by fixing $h_0>0$ and choosing sequences of parameters so that we obtain a strictly increasing sequence of approximate entropies converging to the $h_0$. Then we build an auxiliary map on the Gehman dendrite. 

On each floor, we define a map that cyclically permutes the copies of the tree forming the floor and applies an exact Markov tree map (provided by earlier works) to the last copy before cycling back. This gives a continuous map on the
 set $\Geh\setminus \End(\Geh)$
which extends to the endpoints by continuity. The extended map restricted to $\End(\Geh)$ becomes the dyadic adding machine, which has zero entropy.
By our choice of parameters, the auxiliary map restricted to each floor (which is an invariant set for the auxiliary map) carries entropy close to the corresponding term in the approximating sequence, so the overall entropy of our auxiliary map equals the target value $h_0$ as the entropy is the supremum of entropies of the invariant subsets of the map. However, this map is not mixing, since each floor is invariant. Before we modify the auxiliary map, we use a classical result to equip each floor with a convex metric in which the auxiliary map has a constant slope on the floor. These convex metrics are then combined into a single convex metric on the whole dendrite $\Geh$ so that $\Geh$ has Hausdorff dimension $1$ with respect to it. This metric will be essential for the entropy upper bound.
In the next step, we modify our auxiliary map to make it mixing.  The idea is to make small, carefully controlled modifications on each floor so that orbits can travel between adjacent floors. On each floor, one small interval in the last tree is remapped so that its image reaches into the next floor below. This is done by stretching the map slightly on a single basic interval so that it overshoots into the adjacent floor. Similarly, on each floor (starting from the second), a small interval is remapped so that its image reaches into the floor above.

The modifications are designed so that the slope increase on the affected intervals is bounded by the slope of the next floor, ensuring that the  Lipschitz constant of the modified map does not exceed the target value $h_0$. This ends the construction. It remains  to check that the map has the desired properties. Since the modified map is Lipschitz with respect to the constructed metric, and the Hausdorff dimension of $\Geh$ with respect to that metric is $1$, we can use a theorem of Misiurewicz to see that the topological entropy of the map is at most the logarithm of the Lipschitz constant, which is bounded above by $h_0$. At the same time, on each floor, the set of points whose orbits never escape that floor still supports a subsystem with entropy at least as large as the corresponding approximating value. Since these values converge to $h_0$ and the entropy of the modified map is at least the supremum of these entropies, the entropy is at least $h_0$. The set of endpoints of the Gehman dendrite is closed and backward  invariant for the modified map. 
So no open set contained initially in $\Geh\setminus\End(\Geh)$ can ever be mapped by our map onto the entire dendrite, meaning that our map is not exact.

The expanding nature of the map forces the images of a nonempty open set to eventually cover a basic interval on some floor. Once that happens, the exit modifications propagate the image to neighbouring floors. Since the original map was exact on each floor, the image eventually covers each floor entirely, and once the floor is contained in the image, it stays inside. Repeating this argument floor by floor (upward and downward) shows that the map is topologically mixing.

\begin{theorem}\label{thm:main}
For every $h_0\in(0,\infty]$ 
there exists a pure mixing Gehman dendrite map $F\colon \Geh\to\Geh$ such that $h(F)=h_0$.   
\end{theorem}
\begin{proof} Fix any $h_0>0$. We divided the proof into smaller steps and claims for easier reference.  
\begin{step} Construction of the auxiliary sequences. 
\end{step}

Consider increasing sequences of positive integers 
$(\ell(k))_{k=1}^\infty$ and $(n(k))_{k=1}^\infty$
such that the associated sequence 
\begin{equation}\label{eq:hk}
h_k=\frac{\ell(k)\log (3)}{2^{n(k)}}    
\end{equation}
is strictly increasing and $h_k\nearrow h_0$ as $k\to\infty$. 

We define the sequence $(m(k))_{k=0}^\infty$ inductively: we set $m(0)=0$ and for $k\ge 1$ we set $m(k)=m(k-1)+n(k)$.

\begin{step}\label{step_2}
Construction of an auxiliary continuous map $G\colon\Geh \to \Geh$. 
\end{step}
First, we construct an auxiliary sequence of exact Markov tree maps $(g_k)_{k=1}^\infty$.    
To this end, for each $k\ge 1 $, we use Theorem \ref{odwzorowanie_drzewo_mala_entropia+} to get an exact Markov map $g_k\colon T^{(n(k))}\to T^{(n(k))}$ such that
\begin{equation*}
   \frac{2^{m(k-1)}\ell(k)\log (3)}{2^{n(k)}}\leq h(g_{k})< 
   \frac{2^{m(k-1)}\ell(k+1)\log (3)}{2^{n(k+1)}}. 
\end{equation*}
 We achieve it by taking $g_k=f^{(n(k))}_{\epsilon,L(k)}$, where $L(k)=2^{m(k-1)}\ell(k)$ and 
$\epsilon=2^{m(k-1)}(h_{k+1}-h_k$).
Hence, for $k\ge 1$ we have
\begin{equation}\label{ineq:hk-ent-bounds}
   h_k\leq \frac{1}{2^{m(k-1)}} h(g_{k})< h_k+(h_{k+1}-h_k).
\end{equation}

For every $k\ge 0$ and for every $\omega\in\{0,1\}^{m(k)}$, we denote by  $\Tom$ the subtree of $\Geh$ spanned by $c_\omega$ and $E=\{\gamma\in\{0,1\}^{m(k+1)}:\gamma\restri m(k)=\omega \}$, where $\gamma\restri m(k)=\omega$ means that $\omega$ is the prefix of $\gamma$ of length $m(k)$. 
We note that $\Tom$ is a homeomorphic copy of $\Tnk$. For $x\in\Tnk$ we write $x^{\omega}$ for the corresponding point in $\Tom$. Note that with this convention we have
that $c^\omega_{\omega'}\in \Tom$ corresponds to the vertex  $c_{\omega\omega'}\in\Geh$, where $\omega\omega'$ stands for the concatenation of $\omega$ and $\omega'$. In particular, we identify the vertex $c_\omega$ of $\Geh$ with the root $c^\omega_\lambda$ of $\Tom$. 
We call the disjoint union
\begin{equation*}
\fTk=\bigcup_{\omega\in\{0,1\}^{m(k)}}  \Tom 
\end{equation*}
the $k$-th \emph{floor} of $\Geh$. To shorten our notation, we will write $\firsttree$ for the word $0^{m(k)}$ that labels the first tree  at the floor $k$ and $\lasttree$ for the word $1^{m(k)}$ that labels the last tree  at the floor $k$ (for every $n\ge 1$ we order objects indexed by $\omega\in\{0,1\}^n$ according to the lexicographic order on $\{0,1\}^n$, see Figure \ref{fig:oplus-action}). To keep the notation consistent, we also set $\mathbf{1}(0)=\mathbf{0}(0)=0^{m(0)}$ to be the empty word $\lambda$.

In particular, for $k=1$ we have that the first floor $\mathcal{T}_1$ consists of exactly one tree $T^{\lambda}=T^{(n(1))}$. We define a map $\hat{g}_1\colon\mathcal{T}_1\to\mathcal{T}_1$ 
putting $\hat{g}_1(x)=g_1(x)$ for every  $x\in T^{\lambda}$. 
For each $k\ge 2$ we have that the $k$-th floor $\fTk$ consists of $2^{m(k)}$ isometric disjoint copies of $\Tnk$ indexed by $\omega\in\{0,1\}^{m(k)}$. We define $\hat{g}_k\colon\mathcal{T}_k\to\mathcal{T}_k$ as follows.  
For every $x\in \Tnk$ and $\omega\in\{0,1\}^{m(k)}$ we set 
\begin{equation*}
\hat{g}_k(x^{\omega})=\begin{cases}
x^{(\omega\oplus 1)},&\text{if $\omega\neq \lasttree
$}, \\
g_k(x)^{\firsttree}, &\text{if $\omega=\lasttree
$}.
\end{cases}    
\end{equation*}

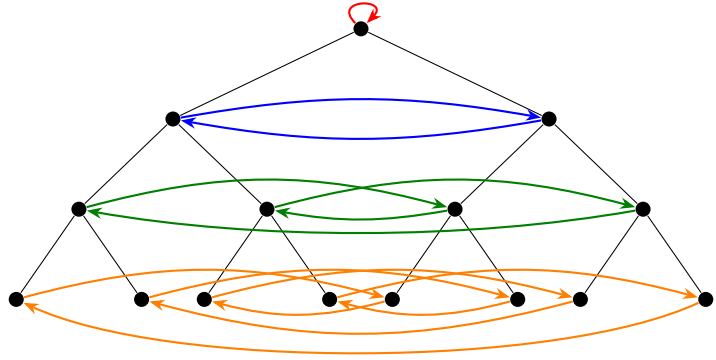
\begin{figure}[H]
    \centering
    \label{here:label}
    \begin{tikzpicture}[
    level/.style={sibling distance=50mm/#1},
    level distance=12mm,
    inner node/.style={circle, fill=black, minimum size=2mm, inner sep=0pt},
    leaf node/.style={circle, fill=black, minimum size=2mm, inner sep=0pt}
]

\node[inner node] (root) {}
    child {node[inner node] (0) {}
        child {node[inner node] (00) {}
            child {node[leaf node] (000) {}}
            child {node[leaf node] (001) {}}
        }
        child {node[inner node] (01) {}
            child {node[leaf node] (010) {}}
            child {node[leaf node] (011) {}}
        }
    }
    child {node[inner node] (1) {}
        child {node[inner node] (10) {}
            child {node[leaf node] (100) {}}
            child {node[leaf node] (101) {}}
        }
        child {node[inner node] (11) {}
            child {node[leaf node] (110) {}}
            child {node[leaf node] (111) {}}
        }
    };


\draw[-{Stealth[length=2mm]}, red, thick, loop above] (root) to [out=135, in=45, looseness=8] (root);

\draw[-{Stealth[length=2mm]}, blue, thick] (0) to [bend left=10] node[midway, above] {} (1);
\draw[-{Stealth[length=2mm]}, blue, thick] (1) to [bend left=10] node[midway, below] {} (0);

\draw[-{Stealth[length=2mm]}, green!50!black, thick] (00) to [bend left=15] node[midway, above] {} (10);
\draw[-{Stealth[length=2mm]}, green!50!black, thick] (10) to [bend left=10] node[midway, above] {} (01);
\draw[-{Stealth[length=2mm]}, green!50!black, thick] (01) to [bend left=15] node[midway, above] {} (11);
\draw[-{Stealth[length=2mm]}, green!50!black, thick] (11) to [bend left=10, looseness=0.8] node[midway, below] {} (00);

\draw[-{Stealth[length=2mm]}, orange, thick] (000) to [bend left=15] node[midway] {} (100);
\draw[-{Stealth[length=2mm]}, orange, thick] (100) to [bend left=15] node[midway] {} (010);
\draw[-{Stealth[length=2mm]}, orange, thick] (010) to [bend left=15] node[midway] {} (110);
\draw[-{Stealth[length=2mm]}, orange, thick] (110) to [bend left=15] node[midway] {} (001);
\draw[-{Stealth[length=2mm]}, orange, thick] (001) to [bend left=15] node[midway] {} (101);
\draw[-{Stealth[length=2mm]}, orange, thick] (101) to [bend left=15] node[midway] {} (011);
\draw[-{Stealth[length=2mm]}, orange, thick] (011) to [bend left=15] node[midway] {} (111);
\draw[-{Stealth[length=2mm]}, orange, thick] (111) to [bend left=25, looseness=0.6] node[midway] {} (000);

\end{tikzpicture}

\caption{
The action of $G$ on the vertices of $T^{(3)}$.
}
\label{fig:periodic-vertices}
\end{figure}

A direct inspection shows that the set of roots of trees in the $k$ floor and the set of all endpoints of these trees, that is, the sets
\begin{align*}
    \{c_{\lambda}^\omega:\omega\in \{0,1\}^{m(k)}\} &=\{c_\gamma\in\Geh: \gamma\in\{0,1\}^{m(k)}\},\\
\{c_{\omega'}^\omega:\omega\in \{0,1\}^{m(k)},\ \omega'\in\{0,1\}^{n(k)}\}&=\{c_\gamma\in\Geh: \gamma\in\{0,1\}^{m(k+1)}
\}
\end{align*}
form two periodic orbits for $\hat{g}_k$ such that $\hat{g}_k(c_\gamma)=c_{\gamma\oplus 1}$ in both cases, see Figure~\ref{fig:periodic-vertices}.  
This observation allows us to see that if for $x\in \Geh\setminus\End(\Geh)=\bigcup_k\fTk$ we set
$G(x)=\hat{g}_k(x)$ for $x\in \fTk$ and $k\geq 1$, then we obtain a well defined and continuous map from $\Geh\setminus\End(\Geh)=\bigcup_k\fTk$ to itself.
We will extend this map to a map $G\colon \Geh\to\Geh$. 
For $x\in \End(\Geh)$ there is a unique sequence $\bar{\omega}=\omega_1\omega_2\omega_3\ldots\in\{0,1\}^\infty$ such that the unique arc in $\Geh$ that joins $c_\lambda$ with $x$ passes through vertices 
\[
c_{\omega_1}, \ c_{\omega_1\omega_2}, \ c_{\omega_1\omega_2\omega_3}, \ldots, c_{\omega_1\ldots \omega_n}, \ldots.
\]
Clearly, the vertices forming this sequence converge to the endpoint, that is
\begin{equation*}
    \lim_{n\to\infty} c_{\omega_1\ldots \omega_n}=c_{\bar{\omega}}. 
\end{equation*}
To be consistent with the definition of $G$ on $\bigcup_k \fTk$, we define
\begin{equation*}
G(c_{\bar{\omega}})= \lim_{n\to\infty} G(c_{\omega_1\ldots \omega_n})=\lim_{n\to\infty} c_{\omega_1\ldots \omega_n\oplus 1}.   
\end{equation*}
Note that the above limit exists, because
\begin{equation*}
    (\omega_1\ldots\omega_n\oplus 1)=\begin{cases}
        0^n,& \text{ if }\bar{\omega}=1^n,\\
        0^{j-1}1\omega_{j+1}\ldots\omega_n,&\text{ where } j=\min\{i\ge 1: \omega_i=0\}.
        \end{cases}
\end{equation*}
Therefore $G(c_{\bar{\omega}})=c_{\alpha(\bar{\omega})}$, where
\begin{equation*}
    \alpha(\bar{\omega})=\begin{cases}
        0^\infty,& \text{ if }\bar{\omega}=1^\infty,\\
        0^{j-1}1\omega_{j+1}\omega_{j+2}\ldots,&\text{ where } j=\min\{i\ge 0: \omega_i=0\}.
    \end{cases}
\end{equation*}
It follows that the map $G$ is continuous and $G|_{\End(\Geh)}$ is conjugated to the dyadic adding machine.

From now on, let $G\colon \Geh\to \Geh$ be the map constructed above. 

\begin{claim}\label{claim:ent-of-gk}
For every $k\ge 1$ we have that $h(G|_{\fTk})=h(\hat{g}_k)$ satisfies
\begin{equation}\label{ineq:hat-gk-ent-bounds}
   h_k\leq h(\hat{g}_{k})< h_k+(h_{k+1}-h_k).
\end{equation}  
\end{claim}
\begin{proof}[Proof of Claim \ref{claim:ent-of-gk}] Fix $k\ge 1$. It is easy to see that for each $\omega\in\{0,1\}^{m(k)}$ the subtree $\Tom$ is invariant for $\hat{g}_{k}^{2^{m(k)}}$, that is,
for every $\omega\in\{0,1\}^{m(k)}$ and $x\in \Tnk$ we have
\[
\hat{g}_{k}^{2^{m(k)}}(x^{\omega})=g_{k}(x)^{\omega}.
\]
It follows that $h(\hat{g}_{k}^{2^{m(k)}})=h(g_{k})$, so $h(\hat{g}_{k})=({2^{m(k)}})^{-1}h(g_{k})$. Using \eqref{ineq:hk-ent-bounds} we get that $h(\hat{g}_{k})$ satisfies \eqref{ineq:hat-gk-ent-bounds}. 
\end{proof}
\begin{claim}\label{claim:ent-of-G}
We have $h(G)=h_0$.    
\end{claim}
\begin{proof}[Proof of Claim \ref{claim:ent-of-G}]
Writing $\Geh=\End(\Geh)\cup\bigcup_k\fTk$ we present $\Geh$ as a union of closed and $G$-invariant sets. 
Since $G|_{\End(\Geh)}$ is the dyadic adding machine, we have $h(G|_{\End(\Geh)})=0$. It follows that
\begin{equation*}
    h(G)=\sup_k h(G|_{\fTk})=\sup_k h(\hat{g}_k).
\end{equation*}
To finish the proof, we combine $h_k\nearrow h_0$ with \eqref{ineq:hat-gk-ent-bounds}.
\end{proof}
\begin{step}
Construction of the special 
convex metric 
$ d_\Geh$ on $\Geh$.    
\end{step}

For each $k\ge 1$ we consider $G_k=G|_{\fTk}$. Let $\fTk'$ be a tree obtained by collapsing  the roots of trees in the $k$ floor, that is points in the set 
\begin{equation*}
    \rts_k=\{c_{\lambda}^\omega:\omega\in \{0,1\}^{m(k)}\} =\{c_\gamma\in\Geh: \gamma\in\{0,1\}^{m(k)}\}
\end{equation*}
to a single point $r_k$. Write $\varphi_k\colon\fTk\to\fTk'$ for the projection map. Since $\rts_k$ is a single periodic orbit for $G_k$, we obtain a factor map $\tilde{G}_k$ on the tree $\fTk'$ with the same entropy as $G_k$. Note that $\tilde{G}_k$ is continuous and $r_k$ is the unique fixed point of $\tilde{G}_k$. Now we invoke \cite[Theorem C]{Misiurewicz_slope}, to get a constant slope map from $\fTk'$ to itself conjugated to $\tilde{G}_k$. With a minor abuse of notation, we denote this map also by $\tilde{G}_k$. Furthermore, the slope $s_k$ of $\tilde{G}_k$ satisfies
\begin{equation*}
 h_k\leq \log s_k < h_k+(h_{k+1}-h_k).    
\end{equation*}
In fact, the topological conjugacy given by \cite[Theorem C]{Misiurewicz_slope}
is the identity between $\fTk'$ with the initial metric $\rho$ and $\fTk'$ endowed with some 
convex metric $d_k$ given by a measure. That is, there is a $\tilde{G}_k$-invariant atomless Borel probability measure $\mu'_k$ on $\fTk'$ such that for $x,y\in \fTk'$ we have $d_k(x,y)=\mu'_k([x,y]_{\fTk'})$, where $[x,y]_{\fTk'}$ stands for the unique arc joining $x$ and $y$ in $\fTk'$. Actually, for $x, y\in \fTk'$ the measure $\mu_k'$ satisfies $\mu_k'([x, y]_{\fTk'})=\rho(\varphi_k(x), \varphi_k(y))$, where $\varphi_k$ is the conjugacy map from \cite[Theorem C]{Misiurewicz_slope}.

Since $\mu_k'$ is atomless and $\varphi_k$ is one-to-one on $\fTk\setminus\rts_k$, we can lift $\mu'_k$ to a measure $\mu_k$ on $\fTk$. As a result, for each $\omega\in\{0,1\}^{m(k)}$ we have a 
convex metric $d_{\Tom}$ on $\Tom$ provided by the measure $\mu_k$.

It is now easy to use the measures $(\mu_k)_{k=1}^\infty$ to find a 
convex metric on $\Geh$. For $x,y\in\Geh$ with $x\neq y$ set $[x,y]_\Geh$ to be the unique arc in $\Geh$ whose endpoints are $x$ and $y$. 
Note that for every finite binary word $\omega$, all endpoints of the tree $\Tom$ are contained in the set of  branch points of $\Geh$. Since the set of branch points of $\Geh$ is countable and each branch point is isolated, there is a collection of arcs $(A_j)_{j\in K}$ such that the arc $A_j$ is  the intersection of $[x,y]_{\Geh}$ with some tree $\Tom$ with $\omega\in\{0,1\}^{m(k(j))}$ and $k(j)$ tells us the floor in which $A_j$ is contained. 
These arcs cover $[x,y]_{\Geh}$ with possible exception of at most two points in $[x,y]_{\Geh}\cap \End(\Geh)$.
Therefore, 
\begin{equation*}
[x,y]_{\Geh}=\overline{\bigcup_{j\in K} A_j},
\end{equation*}
Furthermore, $K$ is finite if and only if $x,y\not\in \End(\Geh)$.

The collection $\{A_j:j\in K\}$ is unique up to enumeration of summands.  Note that if $x,y\in \Geh \setminus \End(\Geh)$, then the set $K$ is finite and taking the closure is not needed.
We define a nonatomic Borel probability measure $\mu_\Geh$ on $\Geh$ to be the convex combination of $\mu_k$'s, that is, 
\begin{equation*}
   \mu_\Geh = \sum_{k=1}^\infty \frac{1}{2^{k}}\mu_{k}.
\end{equation*}
It is straightforward to see that the formula
\[
d_{\Geh}(x,y)=\mu_\Geh([x,y]_{\Geh}) =\sum_{j\in K} \frac{1}{2^{k(j)}}\mu_{k(j)}(A_j),
\]
defines a 
convex metric on $\Geh$ such that for each $k\ge 1$ the map $G_k\colon \fTk\to\fTk$ has the constant slope $s_k$ with respect to $d_{\Geh}$. Therefore $G$ is $\sigma$-linear, expanding and has the slope bounded by $s_0=\log (h_0)$.

\begin{claim} Endowing $\Geh$ with $d_\Geh$ we obtain $\dim_H(\Geh)=1$.
\end{claim}It is clear that for any $\delta>0$ and any $k\ge 1$, the tree $T^{(m(s))}=\bigcup_{k\leq s}\fTk \subseteq \Geh$ has a finite cover $\mathcal{U}_s$ by open sets $U$ satisfying $
\diam U<\delta$ such that $\sum_{U\in \mathcal{U}_s}\diam U<\sum_{k=1}^s \mu_\Geh(\fTk)+1/s<\mu_\Geh(\Geh)+1/s$. But $\bigcup_{k>s}
\fTk$ decomposes into $2^{m(s)}$ disjoint dendrites $\Geh_j$, where $1\le j\le 2^{m(s)}$ (these are homeomorphic copies of $\Geh$) such that for large $s$ each of these copies has $d_\Geh$-diameter smaller than $\delta$. 
Furthermore, we have for each $1\le j\le 2^{m(s)}$ that
\begin{equation*}
\diam(\Geh_j)\le\mu_\Geh(\Geh_j)=\frac{1}{2^{m(s)}}\mu_\Geh(\bigcup_{k>s}\fTk)
=\frac{1}{2^{m(s)}}\sum_{k>s}\frac{1}{2^k}.
\end{equation*}
Hence, taking $s$ large and adding to $\mathcal{U}_s$ open sets obtained by removing the root from each $\Geh_j$ for $1\le j\le 2^{m(s)}$, we get a cover $\mathcal{U}$ of $\Geh$ such that
$$
0<\sum_{U\in \mathcal{U}}\diam U\le\sum_{k=1}^s \mu_\Geh(\fTk)+1/s+\sum_{k>s}\frac{1}{2^k}.
$$
This implies $0<\mathcal{H}^1(\Geh)<\infty$, so $\dim_H(\Geh)=1$.

\begin{step}
Construction of $F$.    
\end{step}
 
We are going to modify  $G$ inductively on each floor $\fTk$ to obtain a new map $F$. Since for each $k\ge 1$ the map $g_k\colon \Tnk\to\Tnk$ 
is Markov, for each $k$ there is a finite set $\Pnk\subseteq \Tnk$ inducing a Markov partition into basic intervals for $g_k$. For every $u\in \Pnk$ and $\omega\in\{0,1\}^{m(k)}$ we  
 use the standard notation $u^\omega$ for the copy of $u$ in $\Tom$ and denote 
\begin{equation*}
\fPk=\bigcup_{\omega\in\{0,1\}^{m(k)}}\{u^\omega:u\in \Pnk\}.    
\end{equation*}
As a result we obtain a partition of the floor $\fTk$ into $\fPk$-basic intervals. 

Before we continue, we introduce  one more piece of notation: Given $k\ge 1$, $\omega\in\{0,1\}^{m(k)}$, and $\omega'\in\{0,1\}^{n(k)}$ we write $B^\omega_{\omega'}$ for the $\fPk$-basic interval  in $\Tom$ whose one endpoint is $c^\omega_{\omega'}$. Similarly, we write $B^\omega_{\lambda}$ for the $\fPk$-basic interval whose one endpoint is $c^\omega_{\lambda}\in\Tom$ which is contained in $[c^\omega_{\lambda},c^\omega_0]$. Note that $c^\omega_{\omega'}$ is then the common endpoint of $B^\omega_{\omega'}$ and $B^{\omega\omega'}_{\lambda}$. 

The first series of modifications results in a map $\tF$. We change $G$ on each floor $k\ge 1$, to get a map $\tF$ with some orbits  travelling down the dendrite (from the floor $k$ to $k+1$). Fix $k\ge 1$. We look at the first tree of the next floor, that is, we consider 
$\Takp\in\fTkp$. It is a copy of $\Tnkp$ attached to $\fTk$ by identifying the leftmost endpoint $c^{\firsttree}_{0^{n(k+1)}}$ of $\Tak$ (recall that $\mathbf{0}(0)=\lambda$) with the root $c_\lambda^{\firsttreeplus}$ of $\Takp$.  For simplicity, we denote the leftmost endpoint $c^{\firsttree}_{0^{n(k+1)}}$ of $\Tak$ by $c^{\firsttree}_{\alpha}$. We note that the arc joining the root of $\Takp$ with its leftmost child in $\Takp$, that is, the arc that joins $c_\lambda^{\firsttreeplus}$ with $c_0^{\firsttreeplus}$ 
contains the $\fPkp$-basic interval $B^{\firsttreeplus}_{\lambda}$ whose endpoint is $c_\lambda^{\firsttreeplus}$.
Since the root of $\Takp$ is $2^{m(k)}$-periodic for $G$ 
and $G_{k+1}|_{\fTkp}=\hat{g}_{k+1}$ is exact  and piecewise expanding on $\fTkp$ there is $j\ge 1$ such that
$B_{\lambda}^{\firsttreeplus}\subseteq G^{j}(B_{\lambda}^{\firsttreeplus})$.  It follows that for infinitely many $N$'s we can find $\wkp^N\in B_{\lambda}^{\firsttreeplus}$ such that $G^N([c_\lambda^{\firsttreeplus},\wkp^N])=B_{\lambda}^{\firsttreeplus}$ and for $0\le j<N$ the set $G^j([c_\lambda^{\firsttreeplus},\wkp^N])$  is contained (not necessarily properly) in at most one $\fPkp$-basic interval.
In particular, $G^{N}(\wkp^N)$ is the endpoint of $B_{\lambda}^{\firsttreeplus}$  other than the root $c_\lambda^{\firsttreeplus}$. Since we may take $N$ to be arbitrarily large, we can also have that $\{G^j(\wkp^N):0\le j< N\}\setminus\fPkp$ is nonempty.

Let $\Tonk$ be the rightmost (last) tree of level $\fTk$. 
By Corollary~\ref{cor:structureTn} we know that there is a point $\pok\in\Tonk$ that divides the free arc 
$[\aok,\bok]$ into two basic intervals and satisfies $G(\pok)=G_k(\pok)=c^{\firsttree}_{\alpha}$. Of course, we also have $\pok,\aok,\bok\in\fPk$.
Consider the basic interval $B^{\firsttree}_{\alpha}=B^{\firsttree}_{0^{n(k+1)}}$ with $c^{\firsttree}_{\alpha}\in B^{\firsttree}_{\alpha}$ and  $B^{\firsttree}_{\alpha}\subseteq G([\aok,\pok])$.

We now redefine $G$ over $[\aok,\pok]\subseteq\Tonk$. First, we need large enough $N$ so that the point $\wkp^N\in J$ is such that
\begin{equation*}
    4d_\Geh(\wkp^N,c_\lambda^{\firsttreeplus})=
4\mu_\Geh([\wkp^N,c_\lambda^{\firsttreeplus}])<(s_{k+1}-s_k)\mu_\Geh([\aok,\pok]) 
\end{equation*}
Note that $s_k\mu_\Geh(A)=\mu_\Geh(G(A))$ for every arc $A\subseteq [\aok,\pok]$. 
Since the metric $d_\Geh$ is geodesic, we find a point 
$\baok 
\in [\aok,\pok]$ such that 
\begin{equation*}
d_\Geh(\baok,\pok)=\frac{d_\Geh([\wkp^N,c_\lambda^{\firsttreeplus}])}{s_k}.
\end{equation*}
To get $\tF$ we modify $G$ only on $[\aok,\pok]$. 

In other words, the modified map $\tF$ agrees with $G$ except that on $[\aok,\pok]$ where our new map  
$\tF$ acts linearly and transforms $[\aok,\baok]$ onto $G([\aok,\pok])\cup[c^{\firsttree}_{\alpha},\wkp^N]$. In particular,   $B^{\firsttree}_{\alpha}, [c^{\firsttree}_{\alpha},\wkp^N]\subseteq \tF([\aok,\baok])$ and maps linearly 
$[\baok,\pok]$  onto $[c^{\firsttree}_{\alpha},\wkp^N]$. 
Observe that the \emph{downstairs exit from $k$ floor}, that is the set  
\begin{equation*}
E_*^k=\tF^{-1}([c^{\firsttree}_{\alpha},\wkp^N])\cap \fTk,    
\end{equation*}
which is an interval contained in $[\aok,\pok]$ whose one endpoint is $\pok$ and the other belongs to $\interior[\aok,\baok]$. 

For $k=1$, we take $\mathcal{P}_1'=\mathcal{P}_1$. 
For $k>1$, let $\fPk'\subseteq\fTk$ be the set 
obtained at the end of the previous step of the construction performed on $\fTkm$.
Recall that $\{G^j(\wkp^N):0\le j< N\}\setminus\fPkp$ is nonempty. 
We let $\fPkp'$ to be $\fPkp\cup \{G^j(\wkp^N):0\le j< N\}$. We let $\fSk=\fPk' \cup \{\baokp\}$.

Note that $\tF$  
is still piecewise linear and expanding. Furthermore,
\begin{equation*}
    \mu_\Geh (\tF([\baok,\pok]))=\mu_\Geh([c^{\firsttree}_{\alpha},\wkp^N])=s_k
\mu_\Geh ([\baok,\pok])
\end{equation*}
and
\begin{align*}
\mu_\Geh (\tF([\aok,\baok])&=\mu_\Geh(G([\aok,\pok]))+\mu_\Geh([c^{\firsttree}_{\alpha},\wkp^N])\\
&=s_k \mu_\Geh([\aok,\baok]))+2\mu_\Geh([c^{\firsttree}_{\alpha},\wkp^N])\\
&<s_k\mu_\Geh ([\aok,\baok])+(s_{k+1}-s_k)\frac{\mu_\Geh ([\baok,\pok])}{2}\\
&< s_{k+1}\mu_\Geh ([\aok,\baok]).
\end{align*}
Therefore, the slope of $\tF|_{\fTk}$  
is bounded by above $s_{k+1}<s_0$. 

Note that $\tF$ is $\sigma$-linear and its intervals of linearity are determined by the set
\begin{equation*}
  \mathcal{S}_\infty=  \bigcup_{k=1}^\infty\fSk.
\end{equation*}
Furthermore, for every $k\ge 1$ we have $\fPk\subseteq\fSk$, $F|_{\fPk}=G|_{\fPk}$,  
and for each $x\in \fSk$ there is $j$ such that $\tF^j(x)\in \fPk\cup \fPkp$.

The second set of changes we apply to $\tF$ obtained from the first series of modifications and all floors $\fTk$ for $k\ge 2$. 

Fix $k\ge 2$. For $k=2$, we take $\mathcal{R}_1'=\mathcal{S}_1$. 
For $k>2$, let $\fRkm'\subseteq\fTkm$ be the set 
obtained at the end of the previous step of the construction performed on $\fTkm$.

Again, we look at the first tree $T^{\firsttreeminus}$ of the previous floor $\fTkm$. 
We consider its leftmost endpoint $c^{\firsttreeminus}_{\alpha}=c^{\firsttree}_\lambda$. 
Let $q^{\lasttree}$ be the point in the rightmost (last) tree $T^{\lasttree}$ of the floor $\fTk$ that belongs to a free interval $[d^{\lasttree},e^{\lasttree}]$ and is mapped by 
\begin{equation*}
G|_{[d^{\lasttree},e^{\lasttree}]}=\tF|_{[d^{\lasttree},e^{\lasttree}]}    
\end{equation*}
onto $c^{\firsttreeminus}_{\alpha}=c^{\firsttree}_\lambda$. 

The endpoint $c^{\firsttreeminus}_{\alpha}=c^{\firsttree}_\lambda$ is 
 the endpoint of the $\fPkm$-basic interval $B^{\firsttreeminus}_{\alpha}\subseteq T^{\firsttreeminus}\subseteq\fTkm$. There are infinitely many $j\ge 1$ such that 
$B^{\firsttreeminus}_{\alpha}\subseteq G^{j}(B^{\firsttreeminus}_{\lambda})$. Therefore, there are  $N\ge 1$ and a point $\vkm^N$ such that $G^N(\vkm^N)$ is the 
endpoint of $B^{\firsttreeminus}_{\alpha}$ other than $c^{\firsttree}_\lambda$ and $\vkm^N$ is sufficiently close to  $c^{\firsttreeminus}_{\alpha}$ to guarantee
\begin{equation*}
    4\mu_\Geh([c^{\firsttreeminus}_{\alpha},\vkm^N])\le (s_{k+1}-s_k)\mu_\Geh([d^{\lasttree},q^{\lasttree}]). 
\end{equation*}
Recall that $G^i(\vkm^N)=F^i(\vkm^N)$ for $i=1,\ldots N$ and that $\fSk\cap (d^{\lasttree},e^{\lasttree})=\emptyset$. Therefore, our later modifications of the map $F$ will not alter the set $\fSk$.
Repeating the arguments and calculations from the first step,
we can modify $\tF|_{\fTk}$ on $[d^{\lasttree},q^{\lasttree}]$. To do so, we find $\bar{d}^\lasttree\in\interior[d^{\lasttree},q^{\lasttree}]$ and increase the slope of $\tF$ on $[d^{\lasttree},\bar{d}^{\lasttree}]$ and $[\bar{d}^{\lasttree},q^{\lasttree}]$ to get $F$ such that $F([d^{\lasttree},\bar{d}^{\lasttree}])$ covers $G([d^{\lasttree},q^{\lasttree}])\cup [c^{\firsttreeminus}_\alpha,\vkm^N]$ and $F([\bar{d}^{\lasttree},q^{\lasttree}])$ covers $[c^{\firsttreeminus}_\alpha,\vkm^N]\subseteq T^{\firsttreeminus}\subseteq\fTkm$. Furthermore, the modified $F$ is linear on $[d^{\lasttree},\bar{d}^{\lasttree}]$ and $[\bar{d}^{\lasttree},q^{\lasttree}]$ with the slope $s$ of the modified map $F$ on each interval of linearity contained in $[d^{\lasttree},q^{\lasttree}]$ satisfying $s_k<s\le s_{k+1}$.

Observe that the \emph{upstairs exit from $k$ floor}, that is the set 
\begin{equation*}
    E^*_k=F^{-1}([c^{\firsttreeminus}_{\alpha},\vkm^N])\cap \fTk
\end{equation*}
is an interval contained in $[d^{\lasttree},q^{\lasttree}]$ whose one endpoint is $q^{\lasttree}$ and the other belongs to $\interior[d^{\lasttree},\bar{d}^{\lasttree}]$. 

Note that $\{G^j(\vkm^N):0\le j< N\}\setminus\fSkm$ is nonempty. 
We let $\fRkm$ to be 
$\fSkm\cup \{G^j(\vkm^N):0\le j< N\}$ and set $\fRk'=\fSk \cup \{\bdokm\}$.

Note that for every $k\ge 1$ the map $F|_{\fTk}$ is piecewise linear, expanding, and its slope is bounded by above $s_{k+1}<s_0$. In particular, $F|_{\fTk}$ is linear on each $\fRk$ basic interval. Furthermore, for every $k\ge 1$ we have $\fPk\subseteq\fSk\subseteq\fRk$,  $F|_{\fPk}=\tF|_{\fPk}=G|_{\fPk}$, $F|_{\fSk}=\tF|_{\fSk}$, and for each $x\in \fRk$ there is $j$ such that $F^j(x)\in \fPkm\cup \fPk\cup \fPkp$.

\begin{claim}\label{claim:upper-bound}
We have $h(F)\le h_0$.    
\end{claim}
\begin{proof}[Proof of Claim \ref{claim:upper-bound}.] If $h_0=\infty$, then there is nothing to prove. Assume $h_0<\infty$. Recall that we endowed $\Geh$ with the 
convex metric $d_\Geh$ such that 
$F$ is piecewise linear on each free arc in $\Geh$ and $\dim_H(\Geh)=1$. Furthermore, the slope on each linearity arc of $F$ is bounded above by $s_0>1$ with $\log (s_0)=h_0$.
Therefore, $F$ is $s_0$-Lipschitz on $\Geh$ and we have $h(F)\le h_0$ by Theorem \ref{lipschitz_entropy_graf}.
\end{proof}
\begin{claim}\label{claim:lower-bound}
For every $k\ge 2$ we have $h(F)\ge h_k$.    
\end{claim}
\begin{proof}[Proof of Claim \ref{claim:lower-bound}.] Fix $k\ge 2$. Let $X_k\subseteq\fTk$ be the set of points whose $F$-orbits never leave the floor $\fTk$. That is,
\begin{equation*}
X_k=\fTk\setminus\bigcup_{j=0}^\infty F^{-j}((\baok,\bbok)) \cup  \bigcup_{j=0}^\infty F^{-j}((\bdok,\beok)).   
\end{equation*}
Then $X_k$ is a forward invariant set such that
$G|_{\fTk}=\hat{g}_k$ is a factor of $F|_{X_k}$. The factor map collapses each connected component of each preimage of the interval $E^*_k$ or $E^k_*$ to a point mapped eventually onto the endpoint $c^{\firsttree}_\lambda$, respectively $c^{\firsttree}_\alpha$.
This yields
\begin{equation*}
h(F)\ge h(F|_{X_k})\ge  h(G|_{\fTk})=h(\hat{g}_k)\ge h_k,    
\end{equation*}
which finishes the proof of the Claim.
\end{proof}
Combining Claims \ref{claim:upper-bound} and \ref{claim:lower-bound} with $h_k\nearrow h_0$ as $k\to\infty$, we get $h(F)=h_0$. 
To finish the proof, we need to prove the last Claim:
\begin{claim}\label{claim:pure-mix}
The map $F$ is pure mixing.    
\end{claim}
\begin{proof}[Proof of Claim \ref{claim:pure-mix}.] Fix a nonempty open set $U\subseteq\Geh$. Clearly, $F^{-1}(\End(\Geh))=\End(\Geh)$. Since $\End(\Geh)$ is closed, the nonempty open set  $\Geh\setminus \End(\Geh)$ satisfies $F(\Geh\setminus \End(\Geh))\subseteq \Geh\setminus \End(\Geh)$, so $F$ cannot be exact. It remains to show that $F$ is mixing. Note that $F$ is $\sigma$-linear with the slope greater or equal to $s_1>1$ on each linearity interval, where $\log (s_1)=h_1$. Therefore, there exists a constant $s>1$ such that for every free arc $I\subseteq\Geh$ we have $\mu_\Geh(F(I))>s\mu_\Geh(I)$ unless there is an endpoint $y$ of a linearity interval of $F$ such that $y\in F(I)$. Since every endpoint of every linearity interval of $F$ eventually belongs to $\fPk$ for some $k$,  there are $k$ and a periodic point $c\in \fPk$ such that for all sufficiently large $m$ the set
$F^m(U)$ contains a point belonging to the orbit of $c$. Then $F^m(U)$ must contain at least one free interval $K$ adjacent to $c$. Without loss of generality, we may assume $K\subseteq\fTk$. 
Now, we look at what happens to the images of $K$ under $F^j$ for large $j$. On one hand, for every $j$ 
we have that $F^j(c)=G^j(c)\in F^j(K)\cap\fPk$. On the other hand, each iterate satisfies $\mu_\Geh(F^{j+1}(K))\ge s \mu_\Geh(F^j(K))$ unless $F^{j+1}(K)$ contains a point from $\fRk$ other than $c$. Hence, there is $j\ge 0$ such that $F^j(K)$ contains an $\fRk$-basic interval
and then, increasing $j$ if necessary, we see that for some $j$ the set $F^j(K)$ contains a $\fPk$-basic interval $J$ for $G|_{\fTk}$. By the definition, for each such $J$ we have $F(J)\supseteq G(J)$, so since
 $G|_{\fTk}$ is exact, the whole floor $\fTk$ will be eventually contained in $F^i(U)$ for all sufficiently large $i$.

Now, for $m$ sufficiently large, $F^m(U)$ contains a point belonging to the orbit of $c\in\fTk\cap\fTkp$ together with a free interval $K\subseteq\fTkp$ adjacent to that point. Repeating the above reasoning, we get $\fTkp$ will be eventually contained in $F^i(U)$ for all sufficiently large $i$. Obviously, 
the same reasoning can be applied to the floor $k-1$ (provided $k\ge 2$).  It follows that for every nonempty open set $U\subset\Geh$ and every $k\ge 1$ there is an integer $j(k)$ such that for every $i\ge j(k)$ we have
 $\fTk\subseteq F^i(U)$. Hence, 
 \begin{equation*}
     \bigcup_{i\ge 0}F^i(U)=\Geh\setminus\End(\Geh),
 \end{equation*}
 which implies that $F$ is topologically mixing.
\end{proof}
This completes the proof of the Main Theorem.
\end{proof}

\section{Acknowledgements} 
This paper is a vastly extended version of  Jakub Tomaszewski's Master thesis completed at Jagiellonian University under supervision of Dominik Kwietniak.

The research project was partly supported by the programme “Excellence initiative --- research university" for AGH University of Krakow (Jakub Tomaszewski) and for Jagiellonian University in Kraków (Dominik Kwietniak).
The research of Piotr Oprocha was partly supported by the project CZ.02.01.01/00/23\_021/0008759 supported by EU funds, through the Operational Programme Johannes Amos Comenius.
Piotr Oprocha is grateful to Max Planck Institute for Mathematics in Bonn for its hospitality and financial support.

\bibliographystyle{amsplain}

\end{document}